\documentclass[a4paper,12pt]{amsart}
\usepackage[colorlinks=true,urlcolor=blue,citecolor=blue,linkcolor=blue,pdfstartview=FitH,pdfpagemode=None,bookmarksnumbered=true]{hyperref}
\usepackage[T1]{fontenc}
\usepackage[cyr]{aeguill}
\usepackage{amsrefs}
\usepackage{graphicx}
\usepackage[frenchb]{babel}

\newtheorem{thm}{Th\'{e}or\`{e}me}[section]
\newtheorem{lem}[thm]{Lemme}

\theoremstyle{definition}

\theoremstyle{remark}
\newtheorem*{rem}{Remarque}
\newtheorem*{expl}{Exemple}

\numberwithin{equation}{section}

\newcommand{\set}[1]{\left\{#1\right\}}
\newcommand{\Real}{\mathbb{R}}

\begin{document}

\title{Point fixe li\'{e} \`{a} une orbite p\'{e}riodique d'un diff\'{e}omorphisme du plan}
\author{Boris Kolev}%
\date{2 avril 1990}%

\address{CMI, 39, rue F. Joliot-Curie, 13453 Marseille cedex 13, France}%
\email{kolev@cmi.univ-mrs.fr}%

\thanks{Je remercie Jean-Marc Gambaudo pour son aide et ses encouragements pr\'{e}cieux durant la pr\'{e}paration de ce travail.}%



\begin{abstract}
\'{E}tant donn\'{e} un diff\'{e}omorphisme $C^{1}$ de $\Real^{2}$ qui poss\`{e}de
une orbite p\'{e}riodique, on montre comment la th\'{e}orie du point fixe de
Nielsen peut \^{e}tre utilis\'{e}e pour \'{e}tablir directement l'existence d'un
point fixe li\'{e} \`{a} cette orbite p\'{e}riodique.
\end{abstract}

\maketitle


\section{Introduction}

Soient $f$ un hom\'{e}omorphisme de $\Real^{2}$ et
$\mathcal{O}=\set{P_{1}, P_{2}, \dotsc, P_{n}}$ une orbite
p\'{e}riodique de $f$, de p\'{e}riode $n$. Supposons que $f$ poss\`{e}de un
point fixe $P_{0}$. On dit que $P_{0}$ \emph{n'est pas li\'{e}} \`{a}
l'orbite p\'{e}riodique $\mathcal{O}$ s'il existe une courbe de Jordan
$\mathcal{C}$, bordant un disque $D$, telle que :
\begin{enumerate}
  \item $\mathcal{O}\subset \mathrm{Int}(D)$ ,
  \item $P_{0}\in \Real^{2}\setminus D$ ,
  \item $f(\mathcal{C})$ est isotope \`{a} $\mathcal{C}$ dans $\Real^{2}\setminus\{P_{0}, P_{1}, \dotsc, P_{n}\}$.
\end{enumerate}
Dans le cas contraire, on dit que $P_{0}$ est \emph{li\'{e}} \`{a}
$\mathcal{O}$.

\begin{expl}
Le point fixe d'une rotation d'angle $2k\pi/n$ est li\'{e} \`{a} n'importe
quelle orbite p\'{e}riodique de cette rotation.
\end{expl}

\begin{figure}
\begin{center}
    \includegraphics{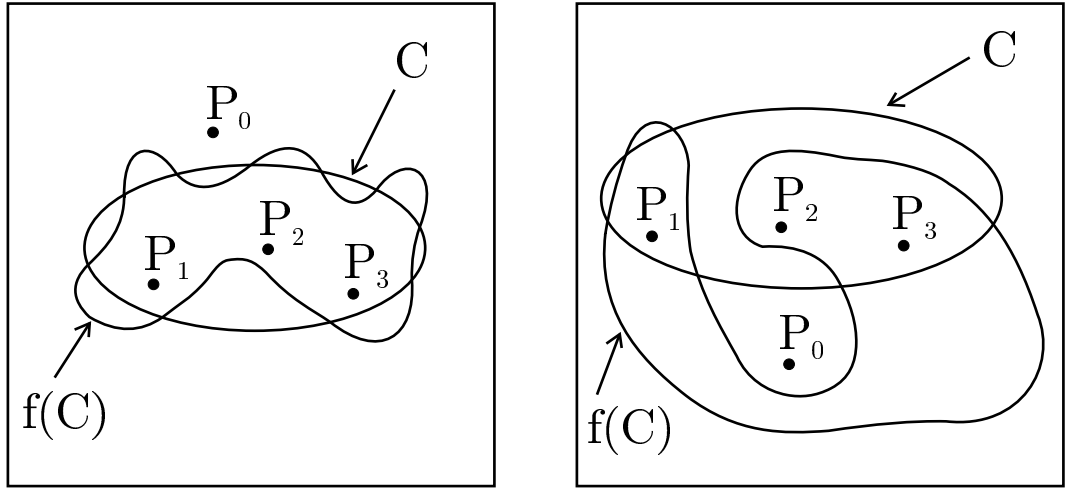}\\
    \caption{A gauche, $\mathcal{C}$ et $f(\mathcal{C})$ sont isotopes dans $\Real^{2}\setminus\{P_{0}, P_{1}, \dotsc, P_{n}\}$, \`{a} droite, elles ne le sont pas.}\label{figure}
\end{center}
\end{figure}

Un des corollaires du th\'{e}or\`{e}me de translation plane de Brouwer
\'{e}nonce que tout hom\'{e}omorphisme de $\Real^{2}$, qui pr\'{e}serve
l'orientation et poss\`{e}de une orbite p\'{e}riodique, a un point fixe
(\cites{Boy89,Fat87}). En utilisant ce r\'{e}sultat et les travaux de
Thurston sur la classification des diff\'{e}omorphismes des surfaces,
J.-M. Gambaudo a montr\'{e}, dans un article r\'{e}cent \cite{Gam90} que
tout diff\'{e}omorphisme $C^{1}$ du disque qui pr\'{e}serve l'orientation et
poss\`{e}de une orbite p\'{e}riodique, poss\`{e}de \'{e}galement un point fixe
\emph{li\'{e}} \`{a} cette orbite p\'{e}riodique.

Il nous est apparu que ce r\'{e}sultat \'{e}tait une cons\'{e}quence directe de
la th\'{e}orie du point fixe de Nielsen. C'est ce que nous proposons de
d\'{e}velopper ici. Plus pr\'{e}cis\'{e}ment, nous montrons le r\'{e}sultat suivant:

\begin{thm}\label{main}
Soit $f:\Real^{2}\rightarrow \Real^{2}$ un diff\'{e}omorphisme $C^{1}$
pr\'{e}servant l'orientation. Si $f$ poss\`{e}de une orbite p\'{e}riodique
$\mathcal{O}$, alors $f$ poss\`{e}de un point fixe li\'{e} \`{a} $\mathcal{O}$.
\end{thm}

\begin{rem}
L'hypoth\`{e}se de diff\'{e}rentiabilit\'{e} est simplement technique et
n'intervient qu'au voisinage des points de l'orbite p\'{e}riodique de
$f$. Une d\'{e}monstration semblable devrait \^{e}tre valable sans cette
hypoth\`{e}se, \textit{voir} \cite[p. 320]{Fat87}.
\end{rem}

\section{Classes de Nielsen}

Nous r\'{e}sumons les \'{e}l\'{e}ments de la th\'{e}orie dont nous avons besoin et
que l'on peut trouver dans \cite{Jia83} et \cite{Nie44}.

Soient $M$ une surface compacte et $\pi : \overline{M}\rightarrow M$
son rev\^{e}tement universel. Consid\'{e}rons un hom\'{e}omorphisme $\tau :
M\rightarrow M$. Deux points fixes de $\tau$, $P_{0}$ et $P_{1}$
sont \emph{Nielsen-\'{e}quivalents} s'il existe un rel\`{e}vement $t$ de
$\tau$ fixant \`{a} la fois un point
$\overline{P}_{0}\in\pi^{-1}(P_{0})$ et un point
$\overline{P}_{1}\in\pi^{-1}(P_{1})$ ou, de fa\c{c}on \'{e}quivalente, s'il
existe un arc $c$ joignant $P_{0}$ \`{a} $P_{1}$ et homotope \`{a} l'arc
$\tau(c)$ relativement \`{a} $P_{0}$, $P_{1}$. Les classes d'\'{e}quivalence
pour cette relation sont les \emph{classes de Nielsen} de $\tau$.
Chaque classe de Nielsen est un sous-ensemble isol\'{e} de
$Fix(\tau)=\set{P\in M;\tau (P)=P}$, on peut donc d\'{e}finir son
indice, qui est un entier relatif, et il existe seulement un nombre
fini de classes ayant un indice non nul. Ce sont les classes
essentielles; leur nombre $N(\tau)$, appel\'{e} nombre de Nielsen de
$\tau$, est un invariant d'homotopie, de sorte que toute application
homotope \`{a} $\tau$ poss\`{e}de au moins $N(\tau)$ points fixes. En
d\'{e}signant ces classes par $F_{1},F_{2},\dotsc ,F_{N}$ et par
$j_{1},j_{2},\dotsc ,j_{N}$ leur indice respectif, on a la relation:
\begin{equation}\label{sum}
    \sum_{k=1}^{N}j_{k}=L(\tau)
\end{equation}
o\`{u}
\begin{equation*}
    L(\tau) = \sum_{i=0}^{2}(-1)^{i} \;\mathrm{tr}\,(\tau_{*i} : H_{i}(M;\mathbb{Q})\rightarrow H_{i}(M;\mathbb{Q}))
\end{equation*}
est le nombre de Lefschetz de $\tau$.

Dans le cas o\`{u} $\chi(M)<0$ ($\chi(M)$ \'{e}tant la caract\'{e}ristique
d'Euler de $M$) et $\tau$ un hom\'{e}omorphisme qui pr\'{e}serve
1'orientation, Nielsen a montr\'{e} que $j_{k}\le 1$ ($k=1, \dotsc,
N$), voir \cite[pp. 17-19]{Nie44}, et c'est l\`{a} le r\'{e}sultat essentiel
dont nous aurons besoin par la suite.

\section{Preuve du th\'{e}or\`{e}me}

Pour commencer, nous identifions $\Real^{2}$ au compl\'{e}mentaire du
p\^{o}le Nord $P_{\infty}$ de la sph\`{e}re unit\'{e} $S^{2}$ gr\^{a}ce \`{a} la
projection st\'{e}r\'{e}ographique. $f$ s'\'{e}tend alors en un hom\'{e}omorphisme
$\hat{f}$ de $S^{2}$ qui fixe $P_{\infty}$.

Soient $\set{P_{1}, P_{2}, \dotsc, P_{n}}$ l'ensemble des points de
l'orbite p\'{e}riodique $\mathcal{O}$ de $f$ et $M$ la surface obtenue
en compactifiant $S^{2}\setminus\{P_{1}, P_{2}, \dotsc, P_{n}\}$ de
la fa\c{c}on suivante. On remplace chaque $P_{i}$ ($i=1, \dotsc, n$) par
un cercle $S_{i}$, le cercle des vecteurs unitaires tangents en ce
point. $\hat{f}$ \'{e}tant de classe $C^{1}$ au voisinage des points
$P_{1},P_{2},\dotsc ,P_{n}$, la restriction de $\hat{f}$ \`{a}
$S^{2}\setminus\set{P_{1}, P_{2}, \dotsc, P_{n}}$ induit un
hom\'{e}omorphisme $\tau : M\rightarrow M$ qui co\"{\i}ncide avec
$\hat{f_{/S^{2}\setminus\set{P_{1},P_{2},\dotsc,P_{n}}}}$ en dehors
du bord de $M$ (\textit{voir} \cite{Bow78}).

Remarquons que $\hat{f}$ et $\tau$ ont les m\^{e}mes points fixes et que
les points fixes de $\hat{f}$, en dehors de $P_{\infty}$, sont
exactement les points fixes de $f$.

\begin{lem}\label{lem1}
Soit $P_{0}$ un point fixe de $f$ non li\'{e} \`{a} $\mathcal{O}$. Alors
$P_{0}$ et $P_{\infty}$ sont dans la m\^{e}me classe de Nielsen de
$\tau$.
\end{lem}

\begin{proof}
$P_{0}$ n'\'{e}tant pas li\'{e} \`{a} $\mathcal{O}$, il existe une courbe de
Jordan $\mathcal{C}$ dans $S^{2}$, s\'{e}parant $\set{P_{0},
P_{\infty}}$ de $\set{P_{1}, P_{2}, \dotsc, P_{n}}$ et telle que
$\hat{f}(\mathcal{C})$ soit isotope \`{a} $\mathcal{C}$ dans
$S^{2}\setminus\set{P_{0}, P_{\infty}, P_{1}, P_{2}, \dotsc,
P_{n}}$.

Cette courbe de Jordan borde, dans $M$, un disque $\Delta$ qui
contient $P_{0}$ et $P_{\infty}$. De m\^{e}me
$\mathcal{C}^{\prime}=\tau(\mathcal{C})$ borde un disque
$\Delta^{\prime}=\tau(\Delta)$ qui contient $\tau(P_{0})=P_{0}$ et
$\tau(P_{\infty})=P_{\infty}$.

$\mathcal{C}$ se rel\`{e}ve donc dans $\overline{M}$, le rev\^{e}tement
universel de $M$, en une courbe ferm\'{e}e $\overline{\mathcal{C}}$
bordant un disque $\overline{\Delta}$ tel que
$\pi(\overline{\Delta})=\Delta$. Par ailleurs, une isotopie
$\mathcal{C}_{s}$ ($s\in[0,1]$) entre $\mathcal{C}$ et
$\mathcal{C}^{\prime}$ se rel\`{e}ve en une isotopie
$\overline{\mathcal{C}}_{s}$ entre $\overline{\mathcal{C}}$ et un
rel\`{e}vement $\overline{\mathcal{C}}^{\prime}$ de
$\mathcal{C}^{\prime}$ bordant un disque
$\overline{\Delta}^{\prime}$ tel que $\pi$
$(\overline{\Delta}^{\prime})=\Delta^{\prime}$.

Soit $t$ l'unique rel\`{e}vement de $\tau$ tel que
$t(\mathrm{\Delta})=\overline{\Delta}^{\prime}$. $\overline{\Delta}$
(resp. $\overline{\Delta}^{\prime}$) contient un unique point
$\overline{P}_{\infty}\in\pi^{-1}(P_{\infty})$ [resp.
$\overline{P}_{\infty}^{\prime}\in\pi^{-1}(P_{\infty})$] et l'on a,
par construction,
$t(\overline{P}_{\infty})=\overline{P}_{\infty}^{\prime}$. Or
$\overline{P}_{\infty}\in\overline{\Delta}^{\prime}$ : sinon il
existerait $s\in[0,1]$ tel que
$\overline{P}_{\infty}\in\overline{\mathcal{C}}_{s}$, ce qui est exclu
puisque l'isotopie $\mathcal{C}_{s}$ \'{e}vite $P_{\infty}$ par
hypoth\`{e}se. Donc $\overline{P}_{\infty}^{\prime} =
\overline{P}_{\infty} = t(\overline{P}_{\infty})$.

Le m\^{e}me argument valant pour $\overline{P}_{0} =
\pi^{-1}(P_{0})\cap\overline{\Delta}$, il en r\'{e}sulte que
$t(\overline{P}_{0})=\overline{P}_{0}$. Donc $P_{0}$ et $P_{\infty}$
sont dans la m\^{e}me classe de Nielsen de $\tau$.
\end{proof}

\begin{lem}\label{lem2}
$N(\tau)\ge 2$.
\end{lem}

\begin{proof}
Remarquons que $L(\tau) = 1 - \mathrm{tr}(\tau_{*1})=2$. Ceci \'{e}tant, nous
distinguons deux cas suivant que la p\'{e}riode $n$ de $\mathcal{O}$ est
\'{e}gale ou strictement sup\'{e}rieure \`{a} $2$.

Si $n=2$, $M$ est un anneau dont $\tau$ permute les bords. En
consid\'{e}rant un rev\^{e}tement \`{a} deux feuillets, on peut montrer par un
argument classique, \textit{voir} \cite{Boy89} et \cite{Fat87}, que
$\tau$ poss\`{e}de exactement deux classes de Nielsen, chacune ayant un
indice \'{e}gal \`{a} $1$.

Si $n\ge 3$, $\chi(M)<0$ et d'apr\`{e}s la remarque faite \`{a} la fin du
paragraphe pr\'{e}c\'{e}dent et la relation \eqref{sum}, on a n\'{e}cessairement
$N(\tau)\ge 2$.
\end{proof}

\begin{proof}[Preuve du Th\'{e}or\`{e}me~\ref{main}]
D'apr\`{e}s le lemme~\ref{lem2}, $\tau$ poss\`{e}de au moins deux points fixes dont
un au moins n'est pas dans la m\^{e}me classe que $P_{\infty}$. Il
r\'{e}sulte alors du lemme~\ref{lem1} que ce point est un point
fixe de $f$ li\'{e} \`{a} $\mathcal{O}$.
\end{proof}


\end{document}